\documentclass[11pt]{amsart}   
\usepackage{amssymb,amscd,latexsym}   
\usepackage{amsmath}
\usepackage{amsthm,times}
\usepackage[margin=1in]{geometry}
\usepackage{tikz}
\usepackage{epsfig,graphicx}
\usepackage{setspace}
\usepackage{nicematrix}

\begin{document}

\newcommand{\mmbox}[1]{\mbox{${#1}$}}
\newcommand{\N}{{\mathbb N}}
\newcommand{\Q}{{\mathbb Q}}
\newcommand{\R}{{\mathbb R}}
\newcommand{\A}{{\mathcal{A}}}
\newcommand{\msm}{\mathop{\rm msm}\nolimits}
\newcommand{\diag}{\mathop{\rm diag}\nolimits}
\newcommand{\1}{\mathop{\bf 1}\nolimits}
\newcommand{\bmat}[1]{\setstretch{1}{\begin{bmatrix}#1\end{bmatrix}}}
\sloppy
\newtheorem{defn0}{Definition}[section]
\newtheorem{prop0}[defn0]{Proposition}
\newtheorem{quest0}[defn0]{Question}
\newtheorem{thm0}[defn0]{Theorem}
\newtheorem{lem0}[defn0]{Lemma}
\newtheorem{corollary0}[defn0]{Corollary}
\newtheorem{example0}[defn0]{Example}
\newtheorem{remark0}[defn0]{Remark}
\newtheorem{prob0}[defn0]{Problem}
\newtheorem{conj0}[defn0]{Conjecture}
\newtheorem{alg0}[defn0]{Algorithm}

\newenvironment{defn}{\begin{defn0}}{\end{defn0}}
\newenvironment{prop}{\begin{prop0}}{\end{prop0}}
\newenvironment{quest}{\begin{quest0}}{\end{quest0}}
\newenvironment{thm}{\begin{thm0}}{\end{thm0}}
\newenvironment{lem}{\begin{lem0}}{\end{lem0}}
\newenvironment{cor}{\begin{corollary0}}{\end{corollary0}}
\newenvironment{exm}{\begin{example0}\rm}{\end{example0}}
\newenvironment{rem}{\begin{remark0}\rm}{\end{remark0}}
\newenvironment{prob}{\begin{prob0}\rm}{\end{prob0}}
\newenvironment{conj}{\begin{conj0}}{\end{conj0}}
\newenvironment{alg}{\begin{alg0}}{\end{alg0}}

\newcommand{\defref}[1]{Definition~\ref{#1}}
\newcommand{\propref}[1]{Proposition~\ref{#1}}
\newcommand{\thmref}[1]{Theorem~\ref{#1}}
\newcommand{\lemref}[1]{Lemma~\ref{#1}}
\newcommand{\corref}[1]{Corollary~\ref{#1}}
\newcommand{\exref}[1]{Example~\ref{#1}}
\newcommand{\secref}[1]{Section~\ref{#1}}
\newcommand{\remref}[1]{Remark~\ref{#1}}
\newcommand{\questref}[1]{Question~\ref{#1}}
\newcommand{\probref}[1]{Problem~\ref{#1}}
\newcommand{\conjref}[1]{Conjecture~\ref{#1}}
\newcommand{\algref}[1]{Algorithm~\ref{#1}}

\title{Primary decompositions via exponent matrices}
\author{Jacob Miller}

\subjclass[2020]{Primary: 13C05 ; Secondary: 13F55.} \keywords{primary decomposition, associated prime, monomial ideal, exponent vector.\\
\indent Author's address: Department of Mathematics and Statistical Science, University of Idaho, Moscow, ID 83844, jaco5247@vandals.uidaho.edu.}

\begin{abstract} In this paper we present a new result that exhibits in a non-recursive fashion a primary decomposition of any monomial ideal. Our result is derived from the classical quotient-sum recursive procedure for determining a primary decomposition of any proper ideal in a commutative Noetherian ring. We also show how our result can be used as an algorithm for producing a primary decomposition directly from the exponent vectors of the minimal generators of the ideal.
\end{abstract}

\date{}
\maketitle

\section{Introduction}

An important result by Lasker (1905) and Noether (1921) says that every proper ideal $I$ in a Noetherian commutative ring $R$ can be written as a finite intersection of primary ideals; this decomposition is called a primary decomposition of $I$. Primary ideals are generalization of irreducible ideals, and Noether proved the above theorem by showing that any proper ideal in a Noetherian ring is the intersection of finitely many irreducible ideals.

The importance of primary decompositions is enormous as almost all quintessential algebraic invariants appearing in Commutative Algebra and Algebraic Geometry are defined from this: (Krull) dimension and heights (codimensions) of ideals or varieties, degree of an ideal or of a variety, and even information about the free resolutions are captured in this concept (vanishing of ${\rm Ext}$ modules in strong correlation to the non-existence of certain associated primes of the corresponding codimensions; see \cite{EiHuVa}), etc.

In general, computing a primary decomposition of an ideal is not straightforward. However, for monomial ideals in polynomial rings, a handful of algorithms exist, and we briefly review these in the next section. Our main result (Theorem \ref{thm-exponent}) leads to an algorithm (Algorithm \ref{alg-exp}) which produces a primary decomposition of a monomial ideal in a non-recursive way, directly from the generators of the ideal, via the corresponding exponent matrix. At the end of the paper, we include a Macaulay2 implementation of the algorithm. 

Note that in \cite[Theorem 3.7 and Section 4]{MiNaTo} the exponent matrices are used to characterize witnesses of monomial ideals and their powers - {\em a witness} of an ideal $I\subset R$ is any element $f\in R$ such that $(I:f)$ is an associated prime of $I$. Their approach is rather theoretical than computational, as the Macaulay2 code presented in their Appendix is just a word-by-word (slow) translation of their Theorem 3.7. In a future project, we plan on investigating if our Theorem \ref{thm-exponent} could produce a more effective (and different) way of finding witnesses of monomial ideals and of any of their powers. 

\section{Preliminaries}

An ideal $Q$ of a commutative ring $R$  is called {\em primary} if $Q$ is proper and, whenever $a,b\in R$ with $ab\in Q$ but $a\notin Q$, then $b\in \sqrt{Q}$. If $Q$ is primary, then $P:=\sqrt{Q}$ is a prime ideal of $R$, and we say that {\em $Q$ is a $P$-primary ideal}. Also, whenever $P'$ is a prime ideal with $Q\subseteq P'$, then $P\subseteq P'$.

Let $R$ be a commutative Noetherian ring and let $I$ be a proper ideal of $R$. Then, $I$ has a {\em primary decomposition} $I=Q_1\cap\cdots\cap Q_n$ (see, \cite[Corollary 4.35]{Sh}), meaning that for all $i=1,\ldots,n$, $Q_i$ is a $P_i$-primary ideal. If, in addition, $P_1,\ldots, P_n$ are all distinct and for all $j=1,\ldots,n$ we have $$Q_j\nsupseteq (Q_1\cap\cdots\cap Q_{j-1}\cap Q_{j+1}\cap\cdots\cap Q_n),$$ then the primary decomposition is called {\em minimal} or {\em irredundant}.

In a minimal primary decomposition as above, the prime ideals $P_i$ are unique, and they are called {\em the associated primes of $I$}; this is denoted by ${\rm Ass}(R/I)=\{P_1,\ldots,P_n\}$. The associated primes that are minimal under inclusion are called the minimal primes of $I$ (see \cite[Proposition 4.24]{Sh}); their set is denoted ${\rm Min}(I)$.

Throughout, we define $R := \mathbb{K}[x_1,\ldots,x_n]$, a polynomial ring over a field $\mathbb{K}$, and $I$ denotes a monomial ideal of $R$ with minimal generating set $G(I)$. It is well known that such a ring $R$ is Noetherian (by Hilbert's Basis Theorem), and that every proper ideal in a Noetherian ring has a primary decomposition (see \cite[Corollary 4.35]{Sh}). A monomial ideal in $R$ is one that is (finitely) generated by monomials in the indeterminates $x_1,\ldots,x_n$. For this special case, several algorithms exist to determine a primary decomposition from the generators of the monomial ideal. We briefly review them as an introduction to our method using exponent matrices. 

\subsection{Irreducible decompositions}

An ideal $I$ of a commutative ring $R$ is called {\em irreducible} if it is a proper ideal and, whenever $I=I_1\cap I_2$ with $I_1, I_2$ ideals of $R$, then $I=I_1$ or $I=I_2$. By \cite[Proposition 4.34]{Sh}, in a commutative Noetherian ring, an irreducible ideal is primary. The converse is not true: for example, in $\mathbb K[x,y]$, the ideal $\langle x^2,y\rangle\cap \langle x,y^2\rangle$ is primary, but not irreducible.

By \cite[Proposition 4.33]{Sh}, every proper ideal in a commutative Noetherian ring can be expressed as a finite intersection of irreducible ideals; because such ideals are primary, this is a primary decomposition. An irreducible decomposition is called {\em minimal} if no factors in the decomposition can be removed to obtain an irreducible decomposition of the same ideal. By \cite[Exercise 7.19]{AtMa}, any minimal irreducible decomposition is unique up to permuting its irreducible factors.

\subsubsection{Splitting algorithm}

Any ideal generated by powers of variables is irreducible, so we may systematically decompose a monomial ideal into an intersection of irreducible ideals as follows (see \cite[Lemma 5.18]{MiSt}): Let $m \in G(I)$. If $m=fg$, where $f$ and $g$ are relatively prime monomials, then
\[
    I = (I+\langle f \rangle) \cap (I+\langle g \rangle ).
\]
Repeating this decomposition on the ideals in the intersection eventually yields an intersection of irreducible ideals, which is a primary decomposition of $I$. This is sometimes called the \emph{splitting algorithm.}

\begin{exm}\label{exm-splitting}
    $I=\langle xy^2,z^3 \rangle = (I + \langle x \rangle) \cap (I + \langle y^2 \rangle) = \langle x,z^3 \rangle \cap \langle y^2,z^3 \rangle$.
\end{exm}\medskip

\subsubsection{Alexander duality}

The theory of Alexander duality links minimal generators of a monomial ideal with its irreducible components. For vectors ${\bf a}=(a_1,\ldots,a_n)$ and ${\bf b}=(b_1,\ldots,b_n)$ in $\N_0^n$ with $b_i \leq a_i$ for $i=1,\ldots,n$, define the new vector ${\bf a\backslash b}$ whose $i$-th component is
\[
    a_i \backslash b_i := \begin{cases}
        a_1 + 1 - b_1 &\text{if $b_i > 0$}\\
        0 & \text{if $b_i=0$.}
    \end{cases}
\]
Suppose $I$ is a monomial ideal and each of its minimal generators divides the monomial $x_1^{a_1}\cdots x_n^{a_n}$. Then the Alexander dual of $I$ with respect to ${\bf a}$ is
\[
    I^{[{\bf a}]} := \bigcap \{ \langle x_1^{a_1 \backslash b_1},\ldots,x_n^{a_n\backslash b_n} \rangle : x_1^{b_1}\cdots x_n^{b_n} \text{ is a minimal generator of $I$}\}.
\]
It can be shown that $(I^{[{\bf a}]})^{[{\bf a}]}=I$, so that (see \cite[Theorem 5.27]{MiSt}),
\[
    I = \bigcap \{ \langle x_1^{a_1 \backslash b_1},\ldots,x_n^{a_n\backslash b_n} \rangle : x_1^{b_1}\cdots x_n^{b_n} \text{ is a minimal generator of $I^{[{\bf a}]}$}\},
\]
which is an irreducible, and hence a primary decomposition of $I$.

\begin{exm}\label{exm-alexanderduality}
    Let $I=\langle x^2y,yz^2 \rangle$ and ${\bf a}=(2,1,2)$. Then $I^{[\bf{a}]} = \langle x,y \rangle \cap \langle y,z \rangle = \langle xz, y \rangle$, so, using the same vector $\bf{a}$ once again, $I = (I^{[{\bf a}]})^{[{\bf a}]} =  \langle x^2,z^2 \rangle \cap \langle y \rangle$ is a primary decomposition of $I$.
\end{exm}\medskip

\subsubsection{Slice algorithm}

A \emph{maximal standard monomial} of $I$ is a monomial $m\not\in I$ such that $mx_i\in I$ for each $i=1,\ldots,n$. The set of maximal standard monomials of $I$ is denoted $\msm(I)$. It turns out (see \cite[Exercise 5.8]{MiSt}) that for a sufficiently large integer $t$, $\msm(I+\langle x_1^t,\ldots,x_n^t \rangle)$ is in bijection with the set of minimal irreducible factors of $I$ via the mapping
\[
    \prod_{i=1}^n x_i^{e_i} \mapsto \sum_{\substack{i=1 \\ e_i+1<t}}^n\langle x_i^{e_i+1}\rangle.
\]

\begin{exm}\label{exm-msm}
    Let $I=\langle x^2,xy \rangle$ and choose $t=3$ ($t$ must be larger than the maximum exponent appearing in the generators of $I$). Then
    \[
        \msm(I+\langle x^3,y^3 \rangle ) = \msm(\langle x^2, xy, y^3 \rangle) = \{x,y^2\},
    \]
    which gets mapped to the set $\{\langle x^2,y \rangle, \langle x \rangle \}$, and, indeed, $I=\langle x^2,y \rangle \cap \langle x \rangle$ is a minimal irreducible decomposition of $I$.
\end{exm}

The \emph{slice algorithm} computes the maximal standard monomials of an ideal by repeatedly ``slicing'' the ideal into smaller pieces, computing the maximal standard monomials of the slices, and reassembling; see \cite{Ro} for more detail.\medskip

\subsection{Quotient algorithm}

Let $I$ be a monomial ideal of $R$ with $G(I)=\{f_1,\ldots,f_s\}$. A primary decomposition of $I$ can be found by the following procedure (see \cite[Algorithm 3.28]{Va}), which we call the \emph{quotient algorithm}. Note that this is the immediate adaptation of the most important method of finding a primary decomposition of any ideal (see \cite[Procedure 3.5]{Sw}).
\begin{itemize}
    \item If there is a simple power of each $x_i$ in $I$, then $I$ is primary.
    \item If $I$ is not primary, let $x_j^d$ be the largest power of $x_j$ such that $x_j\mid f_\ell$ for some $1 \leq \ell \leq s$. Then
    $$I=(I,x_j^d)\cap (I:x_j^d).$$
    \item Repeat with $(I,x_j^d)$ and $(I:x_j^d)$ to get a primary decomposition of $I$.
\end{itemize}

\begin{exm}\label{exm-quotient}
    Let $I=\langle x^2y, xy^3,x^2z,yz^2 \rangle$. Starting with $x^2$, we have
    $$I=(I,x^2)\cap (I:x^2) = \langle x^2,xy^3,yz^2 \rangle \cap \langle y,z \rangle .$$
    Apply the algorithm again to the first ideal, using $y^3$:
    $$I=\langle x^2,y^3,yz^2\rangle \cap \langle x,z^2 \rangle \cap \langle y,z\rangle .$$
    Finally, apply again to the first ideal with $z^2$ to get a primary decomposition of $I$:
    $$I=\langle x^2,y^3,z^2\rangle \cap \langle x^2,y\rangle \cap \langle x,z^2 \rangle \cap \langle y,z\rangle.$$

    This algorithm produces a primary, not necessarily irreducible, decomposition of $I$. Notice that the process is recursive and involves choosing an ideal in the intersection to apply the procedure to next.
\end{exm}\medskip

\section{Main Result}

A finitely-generated monomial ideal in a polynomial ring determines an exponent matrix, which can be transformed, through multiplication by certain diagonal matrices and appending rows, into the exponent matrices of its associated primes. Throughout this section, we use the notation $[n] := \{1,\ldots,n\}$ and, for $m<n$, $[m,n] := \{m,m+1,\ldots,n\}$.

Let $R:=\mathbb K[x_1,\ldots,x_n]$ and let $I=\langle f_1,\ldots,f_k\rangle$ be a monomial ideal minimally generated by $f_1,\ldots,f_k$, where
\[
    f_i = x^{{\bf e}_i} := x_1^{e_{i_1}}\cdots x_n^{e_{i_n}},\quad {\bf e}_i = (e_{i_1},\ldots,e_{i_n}) \in \N_0^n,\quad i = 1,\ldots,k.
\]
Then we may define the $k\times n$ \emph{exponent matrix} of $I$, denoted $E(I)$, by
\[
    E(I) := \begin{bmatrix}
        e_{1_1} & \cdots & e_{1_n} \\
        \vdots & \ddots & \vdots \\
        e_{k_1} & \cdots & e_{k_n}
    \end{bmatrix},
\]
where each row is the exponent vector of a generator of $I$. As the minimal set of generators of a monomial ideal is unique, this matrix is unique up to row order.

It is clear that every $m\times n$ positive integer matrix defines (not uniquely) an ideal in $R$. After we introduce our main result (Theorem \ref{thm-exponent}), we combine it with the exponent matrix to compute the associated primes of a monomial ideal by matrix multiplication and appending rows (Algorithm \ref{alg-exp}).\medskip

The main result is the following modification of the quotient algorithm. It avoids the recursive quality of that algorithm by listing the primary components explicitly (including trivial or redundant ones, which may also appear in the quotient algorithm).

\begin{thm}\label{thm-exponent}
    Let $I$ be a monomial ideal of ${\mathbb K}[x_1,\ldots,x_n]$ with $G(I)=\{x^{{\bf e}_1},\ldots,x^{{\bf e}_k}\}$. Define the vector ${\bf m}=(m_1,\ldots,m_n)$ with entries $m_i:=\max\{e_{1_i},\ldots,e_{k_i}\}$, the largest power of $x_i$ occurring in a generator of $I$. For $S\subseteq [n]$, define ${\bf m}(S)$ entry-wise by
    \[
        {\bf m}(S)_i := \begin{cases}
            m_i &i\in S\\
            0 &i\not\in S.
        \end{cases}
    \]
    Finally, define
    \[
        I_{S}:= (I:x^{{\bf m}(S)})+\langle x_i^{m_i} : i \not\in S\rangle.
    \]
    Then
    \[
        I = \bigcap_{S \subsetneq [n]} I_{S},
    \]
    where each ideal in the intersection is either primary or is the unit ideal.
\end{thm}

In particular, the theorem allows us to find a minimal primary decomposition by removing redundant ideals from the intersection. Before we proceed with a proof, however, we give an example.

\begin{exm}\label{exm-exponent}
    Let $I=\langle x^2y,y^2z\rangle$ be an ideal of $R:=\mathbb{K}[x,y,z]$, so that ${\bf m} := (2,2,1)$. We compute the following:

    \begin{center}
    {\setstretch{1.5}
        \begin{tabular}{lll}
            $S$ & ${\bf m}(S)$ & $I_{S}$ \\
            \hline
            $\emptyset$ & $(0,0,0)$ & $(I:1)+\langle x^2,y^2,z \rangle = \langle x^2,y^2,z \rangle$ \\
            $\{1\}$     & $(2,0,0)$ & $(I:x^2)+\langle y^2,z \rangle = \langle y,z \rangle$ \\
            $\{2\}$     & $(0,2,0)$ & $(I:y^2)+\langle x^2,z \rangle = \langle x^2,z \rangle$ \\
            $\{3\}$     & $(0,0,1)$ & $(I:z)+\langle x^2,y^2 \rangle = \langle x^2,y^2 \rangle$ \\
            $\{1,2\}$   & $(2,2,0)$ & $(I:x^2y^2)+\langle z \rangle = R$ \\
            $\{1,3\}$   & $(2,0,1)$ & $(I:x^2z)+\langle y^2 \rangle = \langle y \rangle$ \\
            $\{2,3\}$   & $(0,2,1)$ & $(I:y^2z)+\langle x^2 \rangle = R$ \\
        \end{tabular}
    }
    \end{center}\medskip
    
    Clearly each ideal in the table is either primary or the unit ideal (the entire ring). The theorem says that
    \[
        I = R \cap \langle y \rangle \cap R \cap \langle x^2,y^2\rangle \cap \langle x^2,z \rangle \cap \langle y,z \rangle \cap \langle x^2,y^2,z \rangle,
    \]
    but as $\langle y \rangle \subseteq \langle y,z \rangle$ and $\langle x^2,y^2 \rangle \subseteq \langle x^2,y^2,z \rangle$, we simplify to
    \[
        I = \langle y \rangle \cap \langle x^2,y^2 \rangle \cap \langle x^2,z \rangle.
    \]
    This is, in fact, a (minimal) primary decomposition of $I$.
\end{exm}\medskip

Now we prove the theorem.

\begin{proof}
    Define
    \[
        J := \bigcap_{S \subsetneq [n]} I_{S}.
    \]
    Our goal is to show that $I=J$ and that each ideal in the intersection is primary or equal to $R$.

    Clearly, since $I\subseteq (I:f)$ for any monomial $f$, we have $I\subseteq J$.
    
    For the other inclusion, we consider $x_1^{c_1}\cdots x_n^{c_n}\not\in I$ and show that $x_1^{c_1}\cdots x_n^{c_n}\notin J$. For at least one $i\in [n]$ we must have $c_i < m_i$, otherwise $x_1^{c_1}\cdots x_n^{c_n}$ is divisible by a generator of $I$. If $c_i < m_i$ for all $i\in [n]$, then $x_1^{c_1}\cdots x_n^{c_n}$ is neither divisible by any generator of $I$ nor divisible by any of $x_1^{m_1},\ldots,x_n^{m_n}$, so, as $I=(I:1)$, we have $x_1^{c_1}\cdots x_n^{c_n}\not\in (I:1)+\langle x_1^{m_1},\ldots,x_n^{m_n}\rangle =: I_{\emptyset} \supset J$ and we are done.
    
    Assume $c_j\geq m_j$ for some $j$. By reordering the variables, we may assume without loss of generality that $c_i<m_i$ for $i=1,\ldots,t$ and $c_j \geq m_j$ for $j=t+1,\ldots,n$. We claim that
    \[
        x_1^{c_1}\cdots x_n^{c_n}\not\in (I:x_{t+1}^{m_{t+1}}\cdots x_n^{m_n})+\langle x_1^{m_1},\ldots,x_t^{m_t}\rangle =: I_{[t+1,n]} \supset J.
    \]
    Indeed, $x_1^{c_1}\cdots x_n^{c_n}x_{t+1}^{m_{t+1}}\cdots x_n^{m_n}$ is not in $I$ as the exponents $c_{t+1},\ldots,c_n$ are already as large as $m_{t+1},\ldots,m_n$, respectively, so increasing these exponents does not affect divisibility by a generator of $I$. Additionally, $x_1^{m_1},\ldots,x_t^{m_t} \nmid x_1^{c_1}\cdots x_n^{c_n}$ since $m_i > c_i$ for $i=1,\ldots,t$. Hence $x_1^{c_1}\cdots x_n^{c_n}\not\in J$, so $I\supset J$, and therefore $I=J$. (Note that this argument shows that not all ideals in the intersection can be the unit ideal, otherwise no choice of monomial outside of $I$ exists!)

    It remains to show that each ideal in the intersection is primary or the unit ideal. If any generator of $I$ divides the monomial $x^{{\bf m}(S)}$, then $(I:x^{{\bf m}(S)})$ is the unit ideal so $I_{S}$ is the unit ideal. If $I_{S}$ is not the unit ideal, then $I_{S}$ is generated by simple powers of the variables $x_i$ for $i\not\in S$ and monomials in these variables, so $I_{S}$ is primary.
\end{proof}\bigskip

Now we return to the exponent matrix, which we use to streamline the computation of the ideals in the intersection given by Theorem \ref{thm-exponent}. Observe that in that intersection, a typical ideal is of the form
\[
    I_S = (I:x^{{\bf m}(S)}) + \langle x_i^{m_i} : i\not\in S \rangle,
\]
and the generators of this ideal are formed by first stripping each generator of $I$ of the variables $x_i$ for $i\in S$, then adding the monomials $x_i^{m_i}$ for $i\not\in S$ and reducing to a minimal generating set. This process is easily performed in terms of the exponent matrix, which we demonstrate with an example.

\begin{exm}\label{exm-exponentmatrix}
    Take $I=\langle x^2y,y^2z \rangle$, the same ideal as in Example \ref{exm-quotient}. Then
    \[
        E(I) = \begin{bmatrix}
            2 & 1 & 0 \\ 0 & 2 & 1
        \end{bmatrix}.
    \]
    To compute the exponent matrix of the ideal in the intersection corresponding to $S=\{2,3\}$, we strip the generators of $I$ of the variables $y$ and $z$ and add the generator $x^2$. In terms of the exponent matrix, this is accomplished by multiplying $E(I)$ by the matrix
    \[
        D_{(1,0,0)} := \diag\{(1,0,0)\} = \begin{bmatrix}
            1 & 0 & 0 \\ 0 & 0 & 0 \\ 0 & 0 & 0
        \end{bmatrix}
    \]
    and appending the row vector $\begin{bmatrix}2 & 0 & 0\end{bmatrix}$. So the exponent matrix of $I_{\{2,3\}}$ is
    \[
        E(I_{\{2,3\}}) = \begin{bNiceMatrix}[columns-width = 10pt]
            \Block{1-3}{E(I) D_{(1,0,0)}}  & & \\
            \quad 2 & 0 & 0 \quad
        \end{bNiceMatrix} = \begin{bmatrix}
            2 & 0 & 0 \\ 0 & 0 & 0 \\ 2 & 0 & 0
        \end{bmatrix},
    \]
    which corresponds to the ideal
    \[
        I_{\{2,3\}} = \langle x^2,1,x^2 \rangle = R,
    \]
    as we saw before.

    Similarly, we compute
    \[
        E(I_{\{1,3\}}) = \begin{bNiceMatrix}[columns-width = 10pt]
            \Block{1-3}{E(I) D_{(0,1,0)}} & & \\
            \quad 0 & 2 & 0 \quad
        \end{bNiceMatrix} = \begin{bmatrix}
            0 & 1 & 0 \\ 0 & 2 & 0 \\ 0 & 2 & 0
        \end{bmatrix},
    \]
    which corresponds to the ideal
    \[
        I_{\{1,3\}} = \langle y,y^2,y^2 \rangle = \langle y \rangle,
    \]
    as before.

    Notice that for some computations we need to append more than one row. For example,
    \[
        E(I_{\{3\}}) = \begin{bNiceMatrix}[columns-width = 10pt]
            \Block{1-3}{E(I) D_{(1,1,0)}} & & \\
            \quad 2 & 0 & 0 \quad \\
            \quad 0 & 2 & 0 \quad
        \end{bNiceMatrix} = \begin{bmatrix}
            2 & 1 & 0 \\ 0 & 2 & 0 \\ 2 & 0 & 0 \\ 0 & 2 & 0
        \end{bmatrix},
    \]
    which gives the ideal
    \[
        I_{\{3\}} = \langle x^2y, y^2, x^2, y^2 \rangle = \langle x^2, y^2 \rangle.
    \]

    Repeating for all possible nonzero diagonal matrices with zeros and ones on the main diagonal gives the same ideals from the intersection in the theorem.
\end{exm}\bigskip

We can now reframe Theorem \ref{thm-exponent} as an algorithm for producing a primary decomposition in terms of exponent matrices.

\begin{alg}\label{alg-exp}
    Let $E(I)$ be the exponent matrix of a monomial ideal $I\subseteq R:=\mathbb{K}[x_1,\ldots,x_n]$ and let $m_i$ denote the largest entry in the $i$-th column for $i=1,\ldots,n$. Form the diagonal matrix
    \[
        \begin{bmatrix}
            {\bf m}_1 \\ \vdots \\ {\bf m}_n
        \end{bmatrix} := \diag(m_1,\ldots,m_n).
    \]
    
    For each proper subset $S\subsetneq [n]$, do the following:
    \begin{enumerate}
        \item Define the vector $\1(S)$ component-wise by
        \[
            \1(S)_i = \begin{cases}
                0 & i\in S\\
                1 & i\not\in S
            \end{cases}
        \]
        and form the diagonal matrix $D_{\1(S)}:= \diag\{\1(S)\}$.
        \item Compute $E(I)D_{\1(S)}$.
        \item For each element $j\in [n]\setminus S$, append the row ${\bf m}_j$ to the matrix from step 2. Denote the resulting matrix by $E(I_S)$.
    \end{enumerate}
    The $2^n-1$ matrices $E(I_S)$ produced by repeating these steps for each proper subset $S$ are the exponent matrices of primary or unit ideals whose intersection is equal to $I$. Reducing this intersection gives a primary decomposition of $I$.
\end{alg}

Note that our definition of $\1(S)$ above is essentially the opposite of our definition of ${\bf m}$ from Theorem \ref{thm-exponent}. This is because taking the quotient by the highest power of a variable is equivalent to zeroing out the corresponding column of the exponent matrix.

We conclude with an example, and note that a Macaulay2 implementation of Algorithm \ref{alg-exp} is included in the appendix.

\begin{exm}\label{exm-algorithm}
    This example illustrates that not every decomposition obtained by this algorithm is irreducible, as was the case in Example \ref{exm-exponent}.

    Let $I=\langle x^2y, x^3z^2,y^2z \rangle$ with
    \[
        E(I) = \begin{bmatrix}
            2 & 1 & 0 \\ 3 & 0 & 2 \\ 0 & 2 & 1
        \end{bmatrix},\quad\text{and}\quad \begin{bmatrix}
            {\bf m}_1 \\ \vdots \\ {\bf m}_n
        \end{bmatrix} = \begin{bmatrix}
            3 & & \\ & 2 & \\ & & 2
        \end{bmatrix}.
    \]

    Then the algorithm produces the following:\medskip
    
    {\small\begin{center}
        \begin{tabular}{c | c c c c c c c}
            $S$
            & $\emptyset$ 
            & $\{1\}$ 
            & $\{2\}$ 
            & $\{3\}$ 
            & $\{1,2\}$ 
            & $\{1,3\}$ 
            & $\{2,3\}$ \\[1.25em]
            
            $D_{\1(S)}$
            & $\bmat{1&&\\&1&\\&&1}$
            & $\bmat{0&&\\&1&\\&&1}$
            & $\bmat{1&&\\&0&\\&&1}$
            & $\bmat{1&&\\&1&\\&&0}$
            & $\bmat{0&&\\&0&\\&&1}$
            & $\bmat{0&&\\&1&\\&&0}$
            & $\bmat{1&&\\&0&\\&&0}$ \\[2.25em]
        
            $E(I_S)$
            & $\bmat{2&1&0\\ 3&0&2\\ 0&2&1\\ 3&0&0\\ 0&2&0\\ 0&0&2}$
            & $\bmat{0&1&0\\ 0&0&2\\ 0&2&1\\ 0&2&0\\ 0&0&2}$
            & $\bmat{2&0&0\\ 3&0&2\\ 0&0&1\\ 3&0&0\\ 0&0&2}$
            & $\bmat{2&1&0\\ 3&0&0\\ 0&2&0\\ 3&0&0\\ 0&2&0}$
            & $\bmat{0&0&0\\ 0&0&2\\ 0&0&1\\ 0&0&2}$
            & $\bmat{0&1&0\\ 0&0&0\\ 0&2&0\\ 0&2&0}$
            & $\bmat{2&0&0\\ 3&0&0\\ 0&0&0\\ 3&0&0}$ \\[4.25em]
        
            ideal
            & $\langle x^3,x^2y,y^2,z^2 \rangle$
            & $\langle y,z^2 \rangle$
            & $\langle x^2,z \rangle$
            & $\langle x^3,x^2y,y^2 \rangle$
            & $R$
            & $R$
            & $R$
        \end{tabular}
    \end{center}}\bigskip
    
    So a primary (although neither irreducible nor irredundant) decomposition of $I$ is
    \[
        I = \langle x^3,x^2y,y^2,z^2 \rangle \cap \langle y,z^2 \rangle \cap \langle x^2,z \rangle \cap \langle x^3,x^2y,y^2 \rangle.
    \]
    Since $\langle x^3,x^2y,y^2 \rangle \subseteq \langle x^3,x^2y,y^2,z^2 \rangle$, we can reduce to a minimal primary decomposition
    \[
        I = \langle y,z^2 \rangle \cap \langle x^2,z \rangle \cap \langle x^3,x^2y,y^2 \rangle.
    \]
\end{exm}

\vskip.2in

\noindent {\bf Acknowledgment:} I would like to thank my PhD advisor, Dr. \c{S}tefan Toh\v{a}neanu, for his guidance throughout this project. I am very grateful to Dr. Mehrdad Nasernejad for all the comments, corrections, and great suggestions on the early stages of the manuscript. All computations were performed using Macaulay2, a (free) software system for research in algebraic geometry (\cite{GrSt}). \medskip

\noindent\textbf{AI Declaration:} Except for Google Gemini which was used to correct and optimize the lines of the Macaulay2 code presented in the Appendix, the author declares that no generative artificial intelligence (GAI) tools were used anywhere else in the preparation, writing, analysis, or publication of this manuscript.\bigskip


\renewcommand{\baselinestretch}{1.0}
\small\normalsize 

\bibliographystyle{amsalpha}
\newpage

\section{Appendix: A Macaulay2 implementation}

The function \verb|computeIntersection| below implements Algorithm \ref{alg-exp} in Macaulay2, and reduces the resulting list of ideals by eliminating redundancies.\bigskip

{\footnotesize\begin{verbatim}
-- exponentMatrx: given an ideal I, return its exponent matrix E.
exponentMatrix = I -> matrix flatten apply(flatten entries mingens I, exponents)

-- maxColumnEntries: given an mxn matrix A, return an nxn diagonal matrix B whose
--                   (i,i) entry is the maximum entry in the ith column of A.
maxColumnEntries = A -> diagonalMatrix apply(entries transpose A, max)

-- properSubsets: given an integer n, return all proper subsets of {0,...,n-1}.
properSubsets = n -> drop(subsets(n),-1)

-- oneS: after defining an integer n and a subset S of {0,...,n-1}, create a diagonal
--       matrix whose (i,i) entry is 0 if i is in S and 1 otherwise.
oneS = (S,n) -> diagonalMatrix(
    toList apply(0..n-1, i -> if isMember(i,S) then 0 else 1)
)

-- matrixToIdeal: given an exponent matrix, produce and trim a monomial matrix.
matrixToIdeal = (M,R) -> monomialIdeal(apply(entries M, v -> R_v))

-- irredundantIdeals: given a list L of ideals, reduce it to a minimal intersection.
irredundantIdeals = L -> (select(L, J -> not any(L, I -> I != J and isSubset(I, J))))

-- computeIntersection: given a monomial ideal I in a polynomial ring R, compute the
                        list of 2^n-1 primary or unit ideals which interesect to give
                        back I, then eliminate redundant ideals.
computeIntersection = (I,R) -> (
    E := exponentMatrix I;
    M := maxColumnEntries E;
    n := numcols M;
    L := toList(0..(n-1));
    S := properSubsets n;
    D := apply(S, s -> E*oneS(s,n) || M^(toList(set L - set s)));
    Q := apply(D, d -> matrixToIdeal(d,R));
    irredundantIdeals Q
)
\end{verbatim}}

\end{document}